\documentclass[12pt]{article}

\usepackage{amsmath, amsthm, amsfonts, amssymb}
\usepackage[utf8]{inputenc}
\usepackage[english]{babel}

\usepackage[colorlinks, linkcolor=blue, citecolor=blue, urlcolor=blue, breaklinks]{hyperref}

\usepackage[T1, T5]{fontenc}

\usepackage[a4paper, margin=2.7cm]{geometry}

\usepackage{cite}
\usepackage{natbib}
\usepackage{url}
\newcommand{\ap}[1]{\left\langle#1\right\rangle}

\def\R{\mathbb{R}}

\def\d{\,\mathrm{d}}
\def\dx{\,\mathrm{d}x}
\def\ddt{\frac{\mathrm{d}}{\mathrm{d} t}}

\def\:{\colon}

\DeclareMathOperator{\sign}{sign}

\newtheorem{thm}{Theorem}[section]

\newtheorem{lem}[thm]{Lemma}
\newtheorem{prp}[thm]{Proposition}
\newtheorem{hyp}{Hypothesis}[section]
\theoremstyle{definition}

\theoremstyle{remark}
\newtheorem{rem}[thm]{Remark}

\title{Close-to-equilibrium behaviour of quadratic reaction-diffusion
  systems with detailed balance}

\date{February 2017}

\author{María J.~Cáceres \and José A.~Cañizo}

\AtEndDocument{\bigskip\bigskip{\footnotesize
  \noindent
  \textsc{María J.~Cáceres, Departamento de Matemática Aplicada,
    Universidad de Granada, 18071 Granada, Spain.}
  \textit{E-mail address}: \texttt{\href{mailto:caceresg@ugr.es}{caceresg@ugr.es}}
  
  \medskip
  \noindent
  \textsc{José A.~Cañizo, Departamento de Matemática Aplicada,
    Universidad de Granada, 18071 Granada, Spain.}
  \textit{E-mail address}:
  \texttt{\href{mailto:canizo@ugr.es}{canizo@ugr.es}}
}}

\begin{document}

\maketitle

\begin{abstract}
  We study general quadratic reaction-diffusion systems with detailed
  balance, in space dimension $d \leq 4$. We show that
  close-to-equilibrium solutions (in an $L^2$ sense) are regular for
  all times, and that they relax to equilibrium exponentially in a
  strong sense. That is: all detailed balance equilibria are
  exponentially asymptotically stable in all $L^p$ norms, at least in
  dimension $d \leq 4$. The results are given in detail for the
  four-species reaction-diffusion system, where the involved constants
  can be estimated explicitly. The main novelty is the regularity
  result and exponential relaxation in $L^p$ norms for $p > 1$, which
  up to our knowledge is new in dimensions 3 and 4.
\end{abstract}

\tableofcontents

\section{Introduction}
\label{sec:intro}

Systems of reaction-diffusion equations model a wide variety of
phenomena, and prominent among them is the behaviour of reacting
chemical mixtures. The mathematical theory of these systems is far
from complete, and in particular the existence of global regular
solutions is unknown in many important cases. If we consider $I \geq 2$ species, denoted
$A_1, \dots, A_I$, undergoing a number $R \geq 1$ of different
reactions
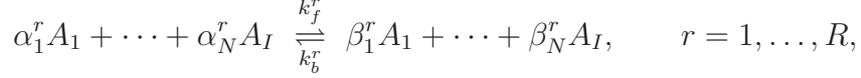
\begin{equation*}
  \alpha_1^r A_1 + \dots + \alpha_N^r A_I
  \
  \underset{k_b^r}{\overset{k_f^r}{\rightleftharpoons}}
  \
  \beta_1^r A_1 + \dots + \beta_N^r A_I,
  \qquad
  r = 1, \dots, R,
\end{equation*}
then the equation satisfied by the concentrations $a_i = a_i(t,x)$ of
the $A_i$ is
\begin{align}
  \label{eq:rd-general-intro}
  &\partial_t a_i = d_i \Delta a_i - R_i(a),
  \qquad
  i = 1, \dots, I,
  \\
  \label{eq:rd-general-boundary-intro}
  &\nabla_x a_i (t,x) \cdot \nu (x) = 0,
  \qquad t > 0, \ x \in \partial \Omega, \ i = 1, \dots, I.
\end{align}
The positive numbers $k_f^r$ and $k_b^r$, for $r=1 \dots, R$, denote
the forward and backward reaction rates, respectively, for each of the
$R$ reactions. The vectors
$\alpha^r = (\alpha_1^r, \dots, \alpha_I^r)$ and
$\beta^r = (\beta_1^r, \dots, \beta_I^r)$ are the \emph{stoichiometric
  coefficients} which specify the number of particles of each species
that take part in each reaction, and the \emph{reaction term} $R_i(a)$
depends on $a = (a_i)_{i=1,\dots,I}$ and is obtained from the law of
mass action. Complete details on this setting are given in Section
\ref{sec:general}, and for the moment we omit them for brevity. A very
interesting basic model that presents the main difficulties is the
following, sometimes called the \emph{four species model}: consider a
set of four chemical substances $A_1, A_2, A_3, A_4$ which undergo the
reactions
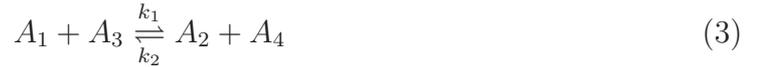
\begin{equation}
  \label{eq:reactions}
  A_1 + A_3
  \underset{k_2}{\overset{k_1}{\rightleftharpoons}}
  A_2 + A_4
\end{equation}
at positive rates $k_1$, $k_2$ as marked. We assume these substances
are confined to a domain $\Omega \subseteq \R^d$ (a connected,
bounded, open region with smooth boundary --- at least
$\mathcal{C}^{2+\alpha}$ with $\alpha > 0$), and we denote the
concentration of $A_i$ by $a_i = a_i(t,x)$, depending on time
$t \geq 0$ and space $x \in \Omega$. Except in Section
\ref{sec:general}, it is understood that the index $i$ always ranges
from $1$ to $4$. If these substances also diffuse in a domain with
diffusion constant $d_i$ corresponding to $A_i$ then the following
system models the time evolution of the concentrations $a_i$:
\begin{equation}
  \label{eq:4spec1}
  \left.
  \begin{aligned}
    \partial_t a_1 &= d_1 \Delta a_1 -k_1 a_1 a_3 + k_2 a_2 a_4,
    \\
    \partial_t a_2 &= d_2 \Delta a_2 + k_1 a_1 a_3 - k_2 a_2 a_4,
    \\
    \partial_t a_3 &= d_3 \Delta a_3 -k_1 a_1 a_3 + k_2 a_2 a_4,
    \\
    \partial_t a_4 &= d_4 \Delta a_4 + k_1 a_1 a_3 - k_2 a_2 a_4,
  \end{aligned}
  \quad \right\}
\qquad t > 0, x \in \Omega.
\end{equation}
We always assume that all $d_i$ are strictly positive. We also assume
no-flux boundary conditions which ensure that the total mass is
conserved:
\begin{equation}
  \label{eq:4spec2}
 \nabla_x a_i (t,x) \cdot \nu (x) = 0,
  \qquad t > 0, \ x \in \partial \Omega, \ i = 1, \dots, 4,
\end{equation}
where $\nu(x)$ denotes the outer normal to the boundary of $\Omega$ at
point $x$. The system \eqref{eq:4spec1} is quadratic in the
nonlinearities and satisfies the \emph{detailed balance condition}:
there is a spa\-ce-homo\-ge\-ne\-o\-u\-s equilibrium
$(a_{i,\infty})_{i=1,\dots,4}$ which makes each of the reactions
balanced, that is, it satisfies
\begin{equation*}
  k_1 a_{1,\infty} a_{3,\infty} = k_2 a_{2,\infty} a_{4,\infty}.
\end{equation*}
Since in this case there is only one reaction, it is obvious that all
space-homogeneous equilibria must satisfy this.  In general, when
detailed balance holds one can show that all equilibria must be
space-homogeneous and satisfy the same condition.

The existence of solutions and asymptotic behaviour of the system
\eqref{eq:4spec1}--\eqref{eq:4spec2} and in general
\eqref{eq:rd-general-intro}--\eqref{eq:rd-general-boundary-intro} have
been studied in a number of works, and several previous results in
reaction-diffusion systems apply to it. In general, difficulties
increase with the strength of the nonlinearities and the space
dimension. It is known that weak solutions to
\eqref{eq:4spec1}--\eqref{eq:4spec2} in $L^2([0,T) \times \Omega)$
exist in all dimensions and for all $T > 0$, and in general weak $L^2$
solutions to
\eqref{eq:rd-general-intro}--\eqref{eq:rd-general-boundary-intro}
exist as long as the system is at most quadratic and satisfies the
detailed balance condition \citep{Desvillettes2007About}. In this
paper we always work with this concept of solution. A general theory
of renormalised solutions for entropy-dissipating systems that does
not have the restriction of the system being quadratic has recently
been developed in \citet{Fischer2015}. Classical solutions are more
elusive: they exist for a short time thanks to general theory of
parabolic equations \citep{amann1985global}, and global-in-time
classical solutions are relatively well understood in a few cases. For
the system \eqref{eq:4spec1}--\eqref{eq:4spec2} the current situation
is the following:

\begin{enumerate}
  
\item Global regular solutions are known to exist in space dimension
  $d \leq 2$ \citep*{Desvillettes2006Exponential,
    Desvillettes2008Entropy, Goudon2010, Canizo2014Improved}.
  
\item Global regular solutions also exist in any space dimension
  whenever the nonlinearities are of degree $< 2$
  \citep{CV2009}. Notice that this does not apply to system
  \eqref{eq:4spec1}, which has quadratic nonlinearities.
  
\item Regular solutions are also understood in any space dimension
  when diffusion coefficients are not too far from each other
  \citep*{Canizo2014Improved,FellnerLatosSuzuki,FellnerLaamri}. They are also
  understood if the diffusion coefficients satisfy $d_1 = d_2$ and
  $d_3 = d_4$ \citep{GZ10,FGZ2014}.
  
\item Finally, there are some special cases where the nonlinearities
  contain some linear terms and allow us to study the regularity of
  some of the $a_i$ separately, and then bootstrap the regularity of
  the other ones (see \citet{Desvillettes2006Exponential}, and
  \citet{Elias2016} for a recent application of this idea).

\end{enumerate}
In general it is not known whether global classical solutions exist
for all smooth initial conditions. Regarding asymptotic behaviour, it
is expected that solutions converge exponentially fast to equilibrium
when detailed balance holds. Entropy methods have been successfully
applied to this kind of systems in almost complete generality,
obtaining exponential convergence to equilibrium in the entropic sense
and in $L^1$ \citep{Desvillettes2006Exponential,
  Desvillettes2008Entropy, GZ10, FellnerTang}. In some cases one may
also obtain $L^2$ convergence by using specifically designed
functionals (see \citet{Rionero2006, Rionero2015} and the references
therein). Relaxation to equilibrium is also understood for
\emph{linear} systems without detailed balance
\citep{Fellner2015Entropy}. Once a rate of relaxation to equilibrium
in entropy has been obtained, any slowly-growing \emph{a priori}
bounds on the solution in $L^p$ spaces or on its regularity yield
convergence to equilibrium in a stronger sense (see
\citet{Desvillettes2008Entropy} and the technique we use to prove
Theorem \ref{thm:exponential-Lp}).

In this paper we study a natural regime which seems to be mostly
missing in the literature: that of close-to-equilibrium
solutions. Some attention has been devoted to this (for example
\citep{Smoller94}) but a general result for detailed balance systems
is, to our knowledge, lacking. Our main result is that
close-to-equilibrium solutions (i.e., whose initial condition is
$L^2$-close to the equilibrium) \emph{in dimension $d \leq 4$} are
regular and converge to equilibrium exponentially fast in all $L^p$
norms including $L^\infty$, which can readily be extended to
exponential convergence in higher Sobolev norms. Notice that closeness
to equilibrium is the main assumption, and no requirements on the
diffusion coefficients $d_i$ are made except for their positivity. The
regularity result is the main novelty, since exponential convergence
to equilibrium can be obtained from it by using simple interpolation
techniques \citep{Desvillettes2006Exponential, FellnerTang}. However,
in the close-to-equilibrium regime we show that this can be obtained
by simpler methods than in the general case. We do not know whether
the limitation that $d \leq 4$ is essential or not.

The natural driving idea is that in the close-to-equilibrium regime
the behaviour is dominated by the linearisation of \eqref{eq:4spec1}
near an equilibrium. We give fully detailed arguments in the case of
the four-species model in Sections
\ref{sec:assumptions}--\ref{sec:regularity}, and later in Section
\ref{sec:general} we extend these ideas to general quadratic equations
with detailed balance.  According to this, the paper is structured as
follows: in Section \ref{sec:assumptions} we gather our precise
assumptions, notation and some preliminary results that are needed
later. In Section \ref{sec:linearised} we briefly study the linearised
system around an equilibrium, mainly gathering results which are
already known in the literature. Section \ref{sec:L2} contains
estimates on the exponential relaxation of the $L^2$ distance to
equilibrium, and these estimates are finally used in Section
\ref{sec:regularity} in order to obtain \emph{a priori} estimates of
higher $L^p$ norms, and strong convergence to equilibrium. The
$L^\infty$ estimates in particular show that close-to-equilibrium
solutions are regular. In Section \ref{sec:general} we state and prove
our theorem on general quadratic reaction-diffusion systems with
detailed balance.

\section{Preliminaries and main results}
\label{sec:assumptions}

\paragraph{Notation.}

Most of the functions we use depend on time $t \geq 0$ and space
$x \in \Omega$. When integrating we often omit the variables $(t,x)$,
and they are assumed unless we explicitly say otherwise. The
differential operators $\nabla$ and $\Delta$ always act on the
variable $x$. When $\Omega$ is regular enough, $\nu(x)$ denotes the
exterior normal to $\Omega$ at a point $x \in \partial \Omega$. As
usual, we denote by $L^p(U)$ the Lebesgue space of real $p$-integrable
functions on a Borel measurable set $U \subseteq \R^m$, for
$1 \leq p \leq \infty$, and the associated norm by $\|\cdot\|_p$. We
consider this norm defined for every measurable function, with its
value possibly being equal to $+\infty$. We will often deal with
vectors $h = (h_1,h_2,h_3,h_4) \in (L^p(\Omega))^4$, with
$\Omega \subseteq \R^d$ an open set, and then we usually denote
\begin{equation}
  \label{eq:Lp}
  \|h\|_p^p := \sum_{i=1}^4 \int_\Omega \frac{|h_i|^p}{a_{i, \infty}^{p-1}} \d x,
  \qquad
  \|\nabla h\|_p^p := \sum_{j=1}^d \sum_{i=1}^4 \int_\Omega
  \frac{|\partial_j h_i|^p}{a_{i, \infty}^{p-1}} \d x,
\end{equation}
where $(a_{i,\infty})_{i=1,\dots,4}$ are numbers defined below, and
represent the equilibrium values of the system
\eqref{eq:4spec1}--\eqref{eq:4spec2}. In the case of $p=2$, the
associated scalar product is
\begin{equation*}
  \ap{h, f} := \sum_{i=1}^4 \int_\Omega \frac{h_i f_i}{a_{i,\infty}}
  \dx
  \qquad \text{for $h, f \in (L^2(\Omega))^4$,}
\end{equation*}
for any $h, f \in (L^2(\Omega))^4$. This weighted definition of
$\|h\|_p$ is especially natural for the $L^2$ norm, which is a
Lyapunov functional for \eqref{eq:4spec1}--\eqref{eq:4spec2}. Any
other weights could be used for $p \neq 2$, but we keep these for
consistency. It should not cause any confusion with the usual notation
for $L^p$ norms, since \eqref{eq:Lp} is only used for vectors in
$(L^p(\Omega))^4$.

If $I \subseteq \R$ is an interval, we denote by $L^p(I, L^q(\Omega))$
the set of measurable functions $u \: I \to L^q(\Omega)$ such that
$\|u\|_q$ is in $L^p(I)$.

\paragraph{Main assumptions.}

Since the reaction constants $k_1, k_2$ are not essential to our
results, let us set $k_1 = k_2 = 1$ throughout to simplify the
notation. Similarly, we assume that $|\Omega| = 1$ (where $|\Omega|$
denotes its Lebesgue measure), in addition to the regularity
hypotheses mentioned in the introduction.  All of our results can
easily be extended to the case of general positive constants
$k_1, k_2$ and a general bounded domain $\Omega$ of any size.
\begin{hyp}
  \label{hyp:Omega}
  The region $\Omega \subseteq \R^d$ is nonempty, open, connected,
  bounded and with boundary of class $\mathcal{C}^{2+\alpha}$ for some
  $\alpha > 0$. We also assume $|\Omega| = 1$.
\end{hyp}

This ensures that the heat equation is well posed in the domain
$\Omega$ with Neumann boundary conditions and avoids non-essential
technicalities. We also need to assume that the diffusion
constants are strictly positive:

\begin{hyp}
  \label{hyp:diffusion>0}
  We assume $d_1, d_2, d_3, d_4$ are strictly positive real numbers.
\end{hyp}

Also, and we always consider nonnegative initial conditions which
ensure that the equilibrium is strictly positive (see below):

\begin{hyp}
  \label{hyp:data}
  The functions $a_{1,0}, a_{2,0}, a_{3,0}, a_{4,0} \: \Omega \to \R$ are in
  $L^2(\Omega)$, are nonnegative, and satisfy
  \begin{equation}
    \label{eq:data-positivity}
    \left( \int_\Omega a_{1,0} \dx \right)
    \left( \int_\Omega a_{3,0} \dx \right)
    +
    \left( \int_\Omega a_{2,0} \dx \right)
    \left( \int_\Omega a_{4,0} \dx \right)
    > 0
  \end{equation}
\end{hyp}
Equivalently: either $a_{1,0}$ and $a_{3.0}$ are both nonzero, or
$a_{2,0}$ and $a_{4,0}$ are both nonzero. This ensures that there are
enough reactants for at least one of the reactions in
\eqref{eq:reactions} to happen, and ensures that the equilibrium is
positive.

With our simplification that $k_1 = k_2 = 1$, equation
\eqref{eq:4spec1} becomes
\begin{equation}
  \label{eq:4spec1-2}
  \left.
  \begin{aligned}
    \partial_t a_1 &= d_1 \Delta a_1 - a_1 a_3 + a_2 a_4,
    \\
    \partial_t a_2 &= d_2 \Delta a_2 + a_1 a_3 - a_2 a_4,
    \\
    \partial_t a_3 &= d_3 \Delta a_3 - a_1 a_3 + a_2 a_4,
    \\
    \partial_t a_4 &= d_4 \Delta a_4 + a_1 a_3 - a_2 a_4,
  \end{aligned}
  \quad \right\}
\qquad t > 0, x \in \Omega,
\end{equation}
always with Neumann boundary conditions
\begin{equation}
  \label{eq:4spec2-2}
 \nabla_x a_i (t,x) \cdot \nu (x) = 0,
  \qquad t > 0, \ x \in \partial \Omega, \ i = 1, \dots, 4,
\end{equation}
and with initial conditions given by the $a_{i,0}$:
\begin{equation}
  \label{eq:4spec2-3}
  a_i(0,x) = a_{i,0}(x)
  \qquad x \in \Omega, \ i = 1, \dots, 4.
\end{equation}
We sometimes denote $a = (a_1, a_2, a_3, a_4)$, we write
$D(a) = (d_i \Delta a_i)_{i=1,\dots,4}$ and denote by
$N(a) = (N_1(a), N_2(a), N_3(a), N_4(a))$ the nonlinear terms on the
right hand side of \eqref{eq:4spec1-2}, so that \eqref{eq:4spec1-2}
can be written as
\begin{equation}
  \label{eq:4spec1-short}
  \partial_t a = D(a) + N(a).
\end{equation}

\paragraph{Concept of solution.}

Take $T \in (0,+\infty]$. As is standard we define a \emph{mild
  solution} (also referred to as \emph{weak solution} by some authors)
of the system \eqref{eq:4spec1-2}--\eqref{eq:4spec2-3} on $[0,T)$ to
be a set of four measurable functions
$a_i \: [0,T) \times \Omega \to \R$, $i=1,\dots,4$, such that the
products $a_1 a_3$ and $a_2 a_4$ are in $L^1([0,t) \times \Omega)$ for
all $0 < t < T$, and which satisfy
\begin{equation*}
  a_i(t,x) = e^{t d_i \Delta} a_{i,0}(t,x)
  + \int_0^t e^{(t-s) d_i \Delta} N_i(a(s,x)) \d s,
  \qquad i = 1,\dots,4
\end{equation*}
almost everywhere on $[0,T) \times \Omega$, where $e^{t \Delta}$
denotes the semigroup associated to the heat equation on $\Omega$ with
Neumann boundary conditions. A \emph{classical
  solution} or \emph{regular solution} is a set of four functions
$a_i \in \mathcal{C}([0,T) \times \overline{\Omega})$,
$i=1,\dots,4$ such that
\begin{equation*}
  \Delta a_i,\ \partial_t a_i \in \mathcal{C}^2((0,T) \times
  \overline{\Omega}),
  \qquad \text{i = 1, \dots, 4},
\end{equation*}
and which satisfy \eqref{eq:4spec1-2}--\eqref{eq:4spec2-3}
pointwise.

\paragraph{Equilibria and conserved quantities.}

The system \eqref{eq:4spec1-2}--\eqref{eq:4spec2-2} has a
three-dimensional space of conserved quantities which is generated for
example by
\begin{equation*}
 M_{12}:= \int_\Omega (a_1 + a_2) \d x,
  \qquad
M_{14}:=  \int_\Omega (a_1 + a_4) \d x
  \quad
\mbox{and} \quad
M_{32}:=  \int_\Omega (a_3 + a_2) \d x.
\end{equation*}
Other conserved quantities can be obtained from these, such as for
example the total mass
\begin{equation*}
  M : =\int_\Omega (a_1 + a_2 + a_3 + a_4) \d x.
\end{equation*}
All of these quantities are formally constant (their time derivative
is $0$). Given nonnegative initial conditions
$a_{i,0} \in L^1(\Omega)$, $i = 1, \dots, 4$, there exists a unique
positive equilibrium which we denote by $a_{i, \infty}$,
$i = 1, \dots, 4$, having the same invariants as the solution; that
is, there is a unique set of four positive numbers
$(a_{i, \infty})_{i=1, \dots, 4}$ such that
\begin{equation}
  \label{eq:equilibrium}
  \begin{gathered}
    a_{1,\infty} a_{3,\infty} = a_{2,\infty} a_{4,\infty},
    \\
    a_{1, \infty} + a_{2,\infty} = \int_\Omega (a_{1,0}
    + a_{2,0}) \d x = M_{12},
    \\
    a_{1, \infty} + a_{4,\infty} = \int_\Omega (a_{1,0}
    + a_{4,0}) \d x = M_{14},
    \\
    a_{2, \infty} + a_{3,\infty} = \int_\Omega (a_{2,0}
    + a_{3,0}) \d x = M_{32}.
  \end{gathered}
\end{equation}
Explicitly,
\begin{equation}
  \label{eq:equilibrium2}
  \begin{gathered}
    a_{1, \infty} = \frac{M_{12}M_{14}}{M},
    \quad a_{2, \infty} = \frac{M_{12}M_{32}}{M},
    \\
     a_{3,\infty} = \frac{M_{32}M_{34}}{M},
     \quad a_{4,\infty} =
    \frac{M_{14}M_{34}}{M}.
  \end{gathered}
\end{equation}

Observe that Hypothesis \ref{hyp:data} ensures that $M \neq 0$, and
that the $a_{i,\infty}$ are strictly positive, since it ensures that
all of $M_{12}$, $M_{14}$, $M_{32}$ and $M_{34}$ are positive.

\paragraph{Main result.}

Our main result for the four-species system can be summarised as
follows:

\begin{thm}
  \label{thm:exponential-Lp}
  Assume Hypotheses \ref{hyp:Omega}, \ref{hyp:diffusion>0} and
  \ref{hyp:data}, and let $d\leq 4$. Let $(a_i)_{i=1,\dots,4}$ be a
  solution to the system \eqref{eq:4spec1-2}--\eqref{eq:4spec2-2}. Let
  $a_{i,\infty}$, $i = 1, \dots, 4$, denote the only positive
  equilibrium of \eqref{eq:4spec1} with the same invariants as
  $(a_{i,0})_{i = 1, \dots, 4}$ (that is, satisfying
  \eqref{eq:equilibrium}).

  Then for any $2 \leq p < \infty$, assuming $\|a_0\|_p < +\infty$,
  there exist positive constants $\lambda, K, \epsilon > 0$ depending
  on $p$, $d$, $(d_i)_{i=1,\dots,4}$,
  $(a_{i,\infty})_{i = 1, \dots, 4}$, $\Omega$ and $\|a_0\|_p$ such
  that
  \begin{equation*}
    \sum_{i=1}^4 \int_\Omega \frac{|a_i(t,x) -
      a_{i,\infty}|^p}{a_{i,\infty}} \dx
    \leq
    K e^{-\lambda t}
  \end{equation*}
  whenever the $a_{i,0}$ satisfy
  \begin{equation*}
    \sum_{i=1}^4 \int_\Omega \frac{|a_{i,0}(x) -
      a_{i,\infty}|^2}{a_{i,\infty}} \dx
    \leq
    \epsilon.
  \end{equation*}
  As a consequence, if $a_{i,0} \in L^\infty(\Omega)$, then the
  solution to \eqref{eq:4spec1-2}--\eqref{eq:4spec2-2} is uniformly
  bounded in $L^\infty(\Omega)$, and is a classical solution for
  $t > 0$.
\end{thm}
Its proof is given at the end of Section \ref{sec:regularity}. We
notice that Theorem \ref{thm:exponential-Lp} contains both a
regularity result and exponential decay of all $L^p$ norms (for
$p < \infty$). As discussed in the introduction, its main novelty is
the regularity of solutions and decay in $L^p$ norms for $p > 1$,
since decay in $L^p$ for $p = 1$ was already known in previous works
\citep{Desvillettes2006Exponential,Desvillettes2008Entropy}.

\medskip

This result can be extended with almost the same proof to a general
quadratic reaction-diffusion system with detailed balance of the form
\eqref{eq:rd-general-intro}--\eqref{eq:rd-general-boundary-intro},
always in dimension $d \leq 4$. Full details are given in Section
\ref{sec:general} and here we just highlight our main result in this
general framework. For the precise assumptions of the above result we
refer to Section \ref{sec:general}, but their essence is that the
system needs to satisfy detailed balance, be at most quadratic, and
otherwise satisfy conditions analogous to Hypotheses
\ref{hyp:Omega}--\ref{hyp:data}:

\begin{thm}
  \label{thm:exponential-Lp-general-intro}
  Assume Hypothesis \ref{hyp:Omega} and Hypotheses
  \ref{hyp:positivity}--\ref{hyp:db-general} (cf.~Section
  \ref{sec:general}), and let $d\leq 4$. Let $(a_i)_{i=1,\dots,I}$ be
  a solution to the system
  \eqref{eq:rd-general-intro}--\eqref{eq:rd-general-boundary-intro}
  with initial condition $a_0 = (a_{0,i})_{i=1, \dots, I}$. Let
  $a_{i,\infty}$, $i = 1, \dots, I$, denote a detailed balance
  equilibrium of \eqref{eq:rd-general-intro} with the same invariants
  as $(a_{i,0})_{i = 1, \dots, I}$.

  Then for any $2 \leq p < \infty$ there exist positive constants
  $\lambda, K, \epsilon > 0$ depending on $p$, $d$, $(d_i)_{i=1,\dots,I}$,
  $(a_{i,\infty})_{i = 1, \dots, I}$, $\Omega$ and $\|a_0\|_p$ such
  that
  \begin{equation*}
    \sum_{i=1}^I \int_\Omega \frac{|a_i(t,x) -
      a_{i,\infty}|^p}{a_{i,\infty}} \dx
    \leq
    K e^{-\lambda t}
  \end{equation*}
  whenever the $a_{i,0}$ satisfy
  \begin{equation*}
    \sum_{i=1}^I \int_\Omega \frac{|a_{i,0}(x) -
      a_{i,\infty}|^2}{a_{i,\infty}} \dx
    \leq
    \epsilon.
  \end{equation*}
  As a consequence, if $a_{i,0} \in L^\infty(\Omega)$, then the
  solution to
  \eqref{eq:rd-general}--\eqref{eq:initial-condition-general} is
  uniformly bounded in $L^\infty(\Omega)$, and is a classical solution
  for $t > 0$.
\end{thm}

The main restrictions and shortcomings of this result are the
following:
\begin{enumerate}
\item As opposed to Theorem \ref{thm:exponential-Lp}, the constant
  $\lambda$ cannot be estimated constructively in Theorem
  \ref{thm:exponential-Lp-general-intro}. Both results rely on a
  spectral gap of the linearised system. In full generality one can
  prove the existence of this positive spectral gap, but we can give
  no estimate on its size. In this generality, estimating
  this spectral gap in any constructive way is a research area by
  itself and is out of the scope of this paper.
  
\item The proof only works in dimension $d \leq 4$, and for reactions
  which involve at most two reactants on each side of the
  reaction. The reason for this, as in the four-species case, is that
  the argument showing the nonlinear perturbations close to
  equilibrium are negligible relies on Sobolev embedding estimates
  which lead to a restriction on the exponents.
  
\item The proof works only for the stability of a detailed balance
  equilibrium, for which the relative entropy (or relative free
  energy) is a known Lyapunov functional of the system. It is possible
  that these results can be extended to systems with complex balance,
  for which there is still an entropy structure, but we have not
  followed that extension.
\end{enumerate}

On the other hand, with the above restrictions the result is fully
general: detailed balance equilibria for quadratic systems in
dimension $d \leq 4$ are always exponentially asymptotically stable,
at least in a neighbourhood of the equilibrium. This applies even if
more than one equilibrium exists, in particular when boundary
equilibria exist (equilibria with some component equal to $0$). This
local stability is known to hold for the system without diffusion
\citep{Horn1972General}, so our contribution is the observation that
diffusion does not change this behaviour for quadratic systems, in
dimension $d \leq 4$.

\paragraph{Entropy.}

Since the system \eqref{eq:4spec1-2} (or \eqref{eq:4spec1}) satisfies
the detailed balance condition, it is well-known that the
\emph{relative free energy} is nonincreasing: if we define it by
\begin{equation*}
  H(a|a_\infty) := \sum_{i=1}^4 \int_\Omega \left(
    a_i \log \frac{a_i}{a_{i,\infty}}
     - a_i + a_{i,\infty}
    \right)
\end{equation*}
then we have
\begin{equation*}
  \ddt H(a|a_\infty) =
  - \sum_{i=1}^4 \int_\Omega d_i \frac{|\nabla a_i|^2}{a_i} \d x
  - \int_\Omega \Psi \left( a_1 a_3, \, a_2 a_4 \right)
  \dx
\end{equation*}
for any solution $a$ to \eqref{eq:4spec1-2}--~\eqref{eq:4spec2-2},
where
\begin{equation*}
  \Psi(a,b) := (a-b) (\log a - \log b)
\end{equation*}
for $a, b > 0$.

\paragraph{Gagliardo-Nirenberg inequalities.}

A basic tool is the so-called \emph{Gagliardo-Nirenberg} inequality,
which we state on a bounded domain $\Omega$, and in the particular
case we will use. This statement can be found in
\citet[pp.~125--126]{nirenberg1959elliptic}; see also \citet[Theorem
9.3]{FriedmanPDE}:

\begin{lem}[Gagliardo-Nirenberg inequality]
  \label{lem:GNS}
  Let $\Omega \subseteq \R^d$ be a nonempty, open, bounded domain with
  Lipschitz boundary. Take $1 \leq q, s < \infty$ and assume $p,
  \theta$ satisfy
  \begin{equation*}
    \frac{1}{p} = \theta \left( \frac{1}{q} - \frac{1}{d} \right)
    + (1-\theta) \frac{1}{s},
    \qquad
    0 \leq \theta < 1,
    \qquad
    1 \leq p < +\infty.
  \end{equation*}
  Then there exists a constant $C > 0$ depending only on $d$, $q$,
  $r$, $p$ and the domain $\Omega$ such that
  \begin{equation}
    \label{eq:GNS}
    \| u \|_p \leq C \|\nabla u\|_q^\theta \| u \|_s^{(1-\theta)} + C\|u\|_1
  \end{equation}
  for all $u \in L^1(\Omega)$.
\end{lem}

\begin{rem}
  Observe that the right hand side, or both sides, may be equal to
  $+\infty$. In the above theorem we have avoided the special cases
  $\theta = 1$ or $q=\infty$, which require additional hypotheses to
  hold and will not be used in this paper. The term $\|u\|_1$ on the
  right hand side is needed due to the fact that we are on a bounded
  domain (otherwise a nonzero constant is a counterexample to the
  inequality). Other versions of this inequality in a bounded domain
  involve the Sobolev norm in $W^{1,q}$ instead of $\|\nabla u\|_q$
  (see \citet[p.~233]{Brezis}), and it is likely that the same
  statement holds without the $\|u\|_1$ term if one writes
  $u - \overline{u}$ instead of $u$ on the left hand side, with
  $\overline{u} := \frac{1}{|\Omega|}\int_\Omega u$. We have not been
  able to find a precise reference that contains these results in full
  generality.
\end{rem}

\paragraph{Parabolic regularisation of the heat equation.}

We will use some relatively well-known results on regularisation in
parabolic equations. For completeness, and since precise references
for some of them are not easily found, we gather them here. We first
state a standard $L^p-L^q$ regularisation property of the heat
equation, and then use it to show $L^\infty$ regularisation of the
heat equation with a source term under certain conditions.

\begin{lem}
  \label{lem:heat}
  Assume $\Omega \subseteq \R^d$ satisfies Hypothesis
  \ref{hyp:Omega}. Consider the solution $u$ to the heat equation on
  $\Omega$ with Neumann boundary conditions and initial condition
  $u_0 \in L^q(\Omega)$, with $1 \leq q < \infty$:
  \begin{equation}
    \label{eq:heat}
    \left.
      \begin{aligned}
        &\partial_t u = \Delta u, \quad && t > 0,\ x \in \Omega,
        \\
        &\nabla u \cdot \nu = 0, \quad && t > 0,\ x \in \partial \Omega
        \\
        &u(0,x) = u_0(x), \quad && x \in \Omega.
      \end{aligned}
      \quad \right\}
  \end{equation}
  For every $r$ with $q \leq r \leq \infty$ it holds that
  \begin{equation*}
    \|u (t,\cdot) \|_r \leq C \|u_0\|_q
    \left(
      1 + t^{-\frac{d}{2} \left(
          \frac{1}{q} - \frac{1}{r} \right)}
      \right)
    \qquad \text{ for all $t > 0$,}
  \end{equation*}
  for some constant $C = C(d,r,q) > 0$.
\end{lem}

\begin{lem}
  \label{lem:heat-regularisation}
  Assume $\Omega \subseteq \R^d$ satisfies Hypothesis \ref{hyp:Omega},
  and take $f \in L^\infty_{\mathrm{loc}}([0,\infty), L^q(\Omega))$ for some
  $1 \leq q \leq \infty$. Consider the solution $u$ to the following
  heat equation on $\Omega$ with a source term $f$ and initial
  condition $u_0 \in L^\infty(\Omega)$:
  \begin{equation}
    \label{eq:heat+source}
    \left.
      \begin{aligned}
        &\partial_t u = \Delta u + f, \quad && t > 0,\ x \in \Omega,
        \\
        &\nabla u \cdot \nu = 0, \quad && t > 0,\ x \in \partial \Omega
        \\
        &u(0,x) = u_0(x), \quad && x \in \Omega.
      \end{aligned}
      \quad \right\}
  \end{equation}
  If $q > d/2$ then
  $u \in L^\infty_{\mathrm{loc}}([0,+\infty), L^\infty(\Omega))$.
\end{lem}

\begin{proof}
  We use Duhamel's formula to write
  \begin{equation*}
    u(t,x) = (e^{t \Delta} u_0)(x) + \int_0^t (e^{(t-s)\Delta} f_s)(x) \d s,
  \end{equation*}
  where $f_s(x) \equiv f(s,x)$ and $e^{t\Delta}$ denotes the heat
  semigroup on $L^1(\Omega)$ at time $t \geq 0$ with Neumann boundary
  conditions. Taking the $L^\infty$ norm and using both Lemma
  \ref{lem:heat} and the maximum principle for the heat semigroup,
  \begin{multline*}
    \|u(t,\cdot)\|_\infty
    \leq
    \|u_0\|_\infty + \int_0^t \| e^{(t-s)\Delta} f_s \|_\infty \d s
    \\
    \leq
    \|u_0\|_\infty + C \int_0^t
    \left(1 + (t-s)^{-\frac{d}{2q}} \right)
    \| f_s \|_q \d s
    \\
    \leq
    \|u_0\|_\infty
    +
    C C_f
    \int_0^t
    \left(1 + (t-s)^{-\frac{d}{2q}} \right)
    \d s,
  \end{multline*}
  where $C_f$ is the norm of $f$ in
  $L^\infty((0,\infty), L^q(\R^d))$. The assumption that $q > d/2$
  shows that the last integral is finite and gives the $L^\infty$
  bound on the solution.
\end{proof}

\section{Linearised system and spectral gap}
\label{sec:linearised}

We begin our study by considering the linearisation of
\eqref{eq:4spec1-2} around the equilibrium
$(a_{i, \infty})_{i=1, \dots, 4}$. If we set
\begin{equation*}
  a_i = a_{i,\infty} + h_i,
  \qquad i = 1,\dots, 4,
\end{equation*}
then the perturbations from equilibrium $h_i$ satisfy the following
system of equations to first order:
\begin{equation}
  \label{eq:4spec1-linearised}
  \left.
  \begin{aligned}
    \partial_t h_1 &= d_1 \Delta h_1 
    - a_{3,\infty} h_1 - a_{1,\infty}h_3 + a_{4,\infty} h_2 + a_{2,\infty} h_4,
    \\
    \partial_t h_2 &= d_2 \Delta h_2 
    + a_{3,\infty} h_1 + a_{1,\infty}h_3 - a_{4,\infty} h_2 - a_{2,\infty} h_4,
    \\
    \partial_t h_3 &= d_3 \Delta h_3 
    - a_{3,\infty} h_1 - a_{1,\infty}h_3 + a_{4,\infty} h_2 + a_{2,\infty} h_4,
    \\
    \partial_t h_4 &= d_4 \Delta h_4 
    + a_{3,\infty} h_1 + a_{1,\infty}h_3 - a_{4,\infty} h_2 - a_{2,\infty} h_4,
  \end{aligned}
  \quad \right\}
\end{equation}
for $t > 0, x \in \Omega$, with the same boundary conditions:
\begin{equation}
  \label{eq:4spec2-linearised}
 \nabla_x h_i (t,x) \cdot \nu (x) = 0,
  \qquad t > 0, \ x \in \partial \Omega, \ i = 1, \dots, 4.
\end{equation}
The conserved quantities of system \eqref{eq:4spec1-2} translate into
the following property for $(h_i)_i$:
\begin{equation}
  \label{eq:h-conservation}
    \int_\Omega (h_1 + h_2) \d x
    =
    \int_\Omega (h_{1} + h_{4}) \d x
    =
    \int_\Omega (h_{2} + h_{3}) \d x
    = 0.
\end{equation}
In particular, $\sum_{i=1}^4 \int_\Omega h_i \dx = 0$. We shorten the
notation of system \eqref{eq:4spec1-linearised} in a similar way as in
\eqref{eq:4spec1-short} by writing
\begin{equation}
  \label{eq:Li}
  \partial_t h_i = d_i \Delta h_i+ L_i h,
  \qquad i = 1, \dots, 4,
\end{equation}
where $h \equiv (h_1, h_2, h_3, h_4)$ and
\begin{equation}
  \label{eq:6}
  L_i h:=(-1)^i
  (a_{3,\infty} h_1 + a_{1,\infty}h_3 - 
  a_{4,\infty} h_2 - a_{2,\infty} h_4).
\end{equation}
Calling $L h := (L_1 h, L_2 h, L_3 h, L_4 h)$ we may write this more
abstractly as
\begin{equation}
  \label{eq:9}
  \partial_t h = D h+L h =: Th.
\end{equation}

Since the equations \eqref{eq:4spec1} satisfy the detailed balance
conditions, the free energy of the system
\eqref{eq:4spec1-2}--\eqref{eq:4spec2-2} is nonincreasing in time and
the system \eqref{eq:4spec1-linearised}--\eqref{eq:4spec2-linearised}
inherits a natural Lyapunov functional: one can check that (with the
definition of $\|h\|_2$ given in \eqref{eq:Lp})
\begin{multline}
  \label{eq:entropy-linearised-2}
  \ddt \| h \|_2^2
  = \ddt \sum_{i=1}^4 \int_\Omega \frac{h_i^2}{a_{i,\infty}} \dx
  = 2 \ap{h, Th}
  \\
  =
  - 2 \sum_{i=1}^4d_i \int_\Omega \frac{|\nabla h_i|^2}{a_{i,\infty}} \d x
  - 2 C_{M}\int_\Omega \left(
    \frac{h_1}{a_{1,\infty}} +
    \frac{h_3}{a_{3,\infty}}
    - \frac{h_2}{a_{2,\infty}}
    - \frac{h_4}{a_{4,\infty}}
  \right)^2 \d x,
\end{multline}
where $C_M=\frac{M_{12}M_{32}M_{14}M_{34}}{M^2}$. We first observe that the
linear operator $T$ satisfies the following inequality, sometimes
known as a \emph{spectral gap inequality}:

\begin{lem}
  \label{lem:gap}
  Take numbers $a_{i,\infty} > 0$ for $i=1,\dots,4$. Assume Hypotheses
  \ref{hyp:Omega} and \ref{hyp:diffusion>0}, and define the linear
  operator $T$ by \eqref{eq:6}--\eqref{eq:9}. There exists
  $\lambda_* > 0$ such that
  \begin{equation*}
    \ap{h, Th} \leq -\lambda_* \|h\|_2^2
  \end{equation*}
  for all $h \in (L^2(\Omega))^4$ satisfying
  \eqref{eq:h-conservation}. The constant $\lambda_*$ depends only on
  the dimension $d$, the domain $\Omega$, the numbers $a_{i,\infty}$,
  $i = 1, \dots, 4$ and the diffusion constants $d_i$,
  $i = 1, \dots, 4$.
\end{lem}

The above result can be deduced from the entropy inequalities in
\citet{Desvillettes2008Entropy}, but we have not been able to find
a precise reference that states it. We give a full proof in our particular
case for completeness:

\begin{proof}[Proof of Lemma \ref{lem:gap}]
  If we denote
  \begin{equation*}
    \bar{h}_i:=\int_\Omega h_i \dx
  \end{equation*}
  then the conservation laws \eqref{eq:h-conservation} allow us to
  write $\bar{h}_i$ in terms of $\bar{h}_1$ as
  \begin{equation*}
    \bar{h}_i=(-1)^{i+1}\bar{h}_1.
  \end{equation*}
  In order to show the result we use the expression of $\ap{h, Th}$
  given in \eqref{eq:entropy-linearised-2}. Due to the Poincar\'e
  inequality we have
  \begin{multline}
    \label{eq:Poincare}
    \sum_{i=1}^4 d_i\int_\Omega
    \frac{\vert \nabla h_i \vert^2}{a_{i,\infty}} \dx
    \geq
    C_\Omega \sum_{i=1}^4 d_i\int_\Omega
    \frac{\vert h_i  - \bar{h}_i \vert^2}{a_{i,\infty}} \dx
    \\
    =
    C_\Omega \sum_{i=1}^4 d_i \int_\Omega
    \frac{\vert h_i \vert^2}{a_{i,\infty}} \dx
    - 
    C_\Omega \sum_{i=1}^4 d_i \frac{\vert \bar{h}_i \vert^2}{a_{i,\infty}},
  \end{multline}
  where $C_\Omega$ is the Poincaré constant of the domain $\Omega$.
  For the second term in \eqref{eq:entropy-linearised-2} it holds that
  \begin{align}
    \int_\Omega \left(
    \frac{h_1}{a_{1,\infty}} +
    \frac{h_3}{a_{3,\infty}}
    - \frac{h_2}{a_{2,\infty}}
    - \frac{h_4}{a_{4,\infty}}
    \right)^2 \d x 
    &\ge
      \left(
      \frac{\bar{h}_1}{a_{1,\infty}} +
      \frac{\bar{h}_3}{a_{3,\infty}}
      - \frac{\bar{h}_2}{a_{2,\infty}}
      - \frac{\bar{h}_4}{a_{4,\infty}}
      \right)^2
      \nonumber
    \\
    &=
      \vert\bar{h}_1\vert^2
      \left(
      \sum_{i=1}^4  \frac{1}{a_{i,\infty}}
      \right)^2.
  \end{align}
  Therefore, taking $\gamma$ small enough so that
  $$
  \gamma
  <
  \min\left\{
    C_\Omega,\
    \frac{C_M\left(\sum_{i=1}^4 \frac{1}{a_{i,\infty}}\right)^2}
    {\sum_{i=1}^4 \frac{d_i}{a_{i,\infty}}}
  \right\}
  $$
  we have
  \begin{align}
    -\ap{h, Th}
    &\ge
      \gamma \left( \sum_{i=1}^4 d_i
      \int_\Omega \frac{\vert h_i \vert^2}{a_{i,\infty}} \dx
      -
      \vert \bar{h}_1 \vert^2\sum_{i=1}^4  \frac{d_i}{a_{i,\infty}}\right)
      +
      C_M\vert\bar{h}_1\vert^2 \left(\sum_{i=1}^4  
      \frac{1}{a_{i,\infty}} \right)^2
                    \nonumber
    \\
                  &\ge \gamma \min_{i\in \{1, \ldots ,4\}}\{d_i\} \Vert h \Vert_2^2
                    +\vert\bar{h}_1\vert^2 \left(
                    C_M\left(\sum_{i=1}^4  
                    \frac{1}{a_{i,\infty}}\right)^2 -
                    \gamma \sum_{i=1}^4  \frac{d_i}{a_{i,\infty}}
                    \right)
                    \nonumber
    \\
                  &\ge \gamma \min_{i\in \{1, \ldots ,4\}}\{d_i\} \Vert h \Vert_2^2,
  \end{align}
  This shows the result.
\end{proof}

\begin{rem}
  Notice that this proof relies on the fact that all constants $d_i$
  are strictly positive (but there is not restriction on how far apart
  they are from each other). It should be possible to adapt it in the
  case in which one of them vanishes by similar arguments to those in
  \citet{DF2007Equadiff}. This and the following Lemma are the main
  points in the paper where the positivity of the $d_i$ is used.
\end{rem}

Using Lemma \ref{lem:gap} in \eqref{eq:entropy-linearised-2} one
readily sees that
\begin{equation*}
  \| h \|_2^2 \leq e^{-2 \lambda_* t} \|h_0\|_2^2,
\end{equation*} for any solution $h$ to \eqref{eq:4spec1-linearised}--\eqref{eq:4spec2-linearised} satisfying \eqref{eq:h-conservation},
where $h_0 := (h_{i,0})_{i = 1, \dots, 4} := (a_{i,0} -
a_{i,\infty})_{i = 1, \dots, 4}$.

In fact, one can find a slightly different inequality that shows
$\ap{h,D h+ Lh}$ actually bounds a stronger norm than $\|h\|_2$:

\begin{lem}
  \label{lem:gap-strong}
  Take numbers $a_{i,\infty} > 0$ for $i=1,\dots,4$. Assume Hypotheses
  \ref{hyp:Omega} and \ref{hyp:diffusion>0}, and define the linear
  operator $T$ by \eqref{eq:6}--\eqref{eq:9}. There exists
  $\lambda > 0$ such that
  \begin{equation*}
    \ap{h, Th} \leq -\lambda \|h\|_2^2 -  \lambda \| \nabla h \|_2^2
  \end{equation*}
  for all $h \in (L^2(\Omega))^4$ satisfying
  \eqref{eq:h-conservation}. The constant $\lambda$ depends only on
  the dimension $d$, the domain $\Omega$, the numbers $a_{i,\infty}$,
  $i = 1, \dots, 4$ and the diffusion constants $d_i$,
  $i = 1, \dots, 4$.
\end{lem}

\begin{proof}
  We set
  \begin{equation*}
    \tilde{L}_i h := d_i\Delta h_i+ L_i h- \frac12 d_i \Delta h_i
    = \frac12 d_i\Delta h_i+ L_i h
  \end{equation*}
  so that
  \begin{equation*}
    2 \langle{ h, \tilde{L} h }\rangle
    \\
    = - \sum_{i=1}^4 \int_\Omega \frac{|\nabla h_i|^2}{a_{i,\infty}}
    \d x - 2 C_M\int_\Omega \left( \frac{h_1}{a_{1,\infty}} +
      \frac{h_3}{a_{3,\infty}} - \frac{h_2}{a_{2,\infty}} -
      \frac{h_4}{a_{4,\infty}} \right)^2 \d x
  \end{equation*}
  One sees that the inequality
  \begin{equation}
    \label{eq:gap-abstract2}
    2 \langle{ h,  \tilde{L} h }\rangle \leq -\tilde{\lambda} \| h \|_2^2,
  \end{equation}
  still holds for some $\tilde{\lambda} > 0$, so that
  \begin{equation*}
    2  \ap{ h, T h} \leq
    -\tilde{\lambda} \| h \|_2^2
    -  \sum_{i=1}^4 d_i\int_\Omega \frac{|\nabla
      h_i|^2}{a_{i,\infty}}\d x
    \leq
    -\tilde{\lambda} \| h \|_2^2
    -  \min_{i=1,\dots,4}\{d_i\} \| \nabla h \|_2^2.
    \qedhere
  \end{equation*}
\end{proof}

We will use this strengthened inequality in the following section.

\section{Local $L^2$ stability}
\label{sec:L2}

We now plan to use Lemma \ref{lem:gap-strong} in order to deduce
$L^\infty([0,\infty), L^2(\Omega))$ bounds on solutions to the
nonlinear system \eqref{eq:4spec1}--\eqref{eq:4spec2-2}, together with
an exponential convergence to equilibrium of solutions:

\begin{thm}
  \label{thm:L2-convergence}
  Assume Hypotheses \ref{hyp:Omega}, \ref{hyp:diffusion>0} and
  \ref{hyp:data}, and assume $d \leq 4$. Let $(a_i)_{i=1,\dots,4}$ be
  a solution to the system
  \eqref{eq:4spec1-2}--\eqref{eq:4spec2-2}. Let $a_{i,\infty}$,
  $i = 1, \dots, 4$, denote the only positive equilibrium of
  \eqref{eq:4spec1-2} with the same invariants as
  $(a_{i,0})_{i = 1, \dots, 4}$ (that is, satisfying
  \eqref{eq:equilibrium}).

  There exist positive constants $\lambda, \epsilon > 0$ depending on
  $d$, $(d_i)_{i=1,\dots,4}$, $(a_{i,\infty})_{i = 1,
    \dots, 4}$ and $\Omega$, such that
  \begin{equation*}
    \sum_{i=1}^4 \int_\Omega \frac{|a_i(t,x) -
      a_{i,\infty}|^2}{a_{i,\infty}} \dx
    \leq
    e^{-\lambda t} \sum_{i=1}^4 \int_\Omega \frac{|a_{i,0}(x) -
      a_{i,\infty}|^2}{a_{i,\infty}} \dx
  \end{equation*}
  whenever the $a_{i,0}$ satisfy
  \begin{equation*}
    \sum_{i=1}^4 \int_\Omega \frac{|a_{i,0}(x) -
      a_{i,\infty}|^2}{a_{i,\infty}} \dx
    \leq
    \epsilon.
  \end{equation*}
\end{thm}

\begin{rem}
  All constants in the above statement can be explicitly given in
  terms of the parameters of the problem, as can be seen from the
  proof.
\end{rem}

\begin{rem}
  In particular we have $L^1$ convergence:
  $\sum_{i=1}^4 \Vert a_i-a_{i,\infty} \Vert_{L^1(\Omega)}^2 \le K
  e^{-\lambda t}$
  for some constant $K$ depending on the parameters of the problem and
  the initial data.
\end{rem}

\begin{proof}[Proof of Theorem \ref{thm:L2-convergence}]
  In this proof we denote by $C, C_1, C_2, \dots$ any number that
  depends only on the parameters in the statement ($d$,
  $(d_i)_{i=1,\dots,4}$, $(a_{i,\infty})_{i = 1, \dots, 4}$ and
  $\Omega$). We deduce an \emph{a priori} estimate on the $L^2$ norm
  of the solution, which can be rigorously justified by a standard
  approximation argument (for example by truncating the nonlinearity
  of the equation).

  Consider a solution $(a_1, a_2, a_3, a_4)$ to the system
  \eqref{eq:4spec1}--\eqref{eq:4spec2-2}.  Defining $h_i$ as before,
  \begin{equation*}
    a_i = a_{i,\infty} + h_i,
    \qquad i = 1,\dots, 4,
  \end{equation*}
  we have that the difference to equilibrium
  $h = (h_1, h_2, h_3, h_4)$ satisfies (using the notation of Section
  \ref{sec:linearised})
  \begin{equation}
    \label{eq:around-eq}
    \left.
      \begin{aligned}
        \partial_t h_1 &= d_1 \Delta h_1 + L_1 h - h_1h_3 + h_2h_4,
        \\
        \partial_t h_2 &= d_2 \Delta h_2 + L_2 h + h_1h_3 - h_2h_4,
        \\
        \partial_t h_3 &= d_3 \Delta h_3 + L_3 h - h_1h_3 + h_2h_4,
        \\
        \partial_t h_4 &= d_4 \Delta h_4 + L_4 h + h_1h_3 - h_2h_4,
      \end{aligned}
      \quad \right\}
    \qquad t \geq 0,\ x \in \Omega,
  \end{equation}
  together with the mass-conserving boundary conditions
  \eqref{eq:4spec2-linearised}. Since the solution
  $(a_i)_{i=1,\dots,4}$ satisfies the conservation laws
  \eqref{eq:equilibrium}, the $h_i$ satisfy \eqref{eq:h-conservation}.
  If we denote the nonlinear terms in the right hand side of
  \eqref{eq:around-eq} by
  $$N(h) = (N_1 h, N_2 h, N_3 h, N_4 h)$$
  we may write \eqref{eq:around-eq} as
  \begin{equation*}
    \partial_t h = D h+ Lh + N(h) = Th + N(h).
  \end{equation*}
  We can now estimate the evolution of the $L^2$ norm defined in
  \eqref{eq:Lp} by using Lemma \ref{lem:gap-strong}:
  \begin{equation}
    \label{eq:1}
    \ddt \| h \|_2^2
    =
    2 \ap{h, Th}
    + 2 \ap{h, N(h)}
    \leq
    - 2 \lambda \|h\|_2^2
    - 2 \lambda \|\nabla h\|_2^2
    + 2 \ap{h, N(h)}.
  \end{equation}
  We use the Gagliardo-Nirenberg inequality in Lemma \ref{lem:GNS} to
  obtain
  \begin{equation*}
    \| h \|_3 \le C_1 \| \nabla h \|_2^\theta 
    \|  h \|_2^{1-\theta}
    + C_1 \|h\|_2, \qquad
    \theta=\frac{d}{6}
  \end{equation*}
  (which is actually valid for $d < 6$). This implies
  \begin{equation*}
    \| h \|_3^{\frac{2}{\theta}}
    \leq
    C_2 \| \nabla h \|_2^2
    \|  h \|_2^{\frac{2(1-\theta)}{\theta}}
    + C_2 \|h\|_2^{\frac{2}{\theta}}.
  \end{equation*}
  that is,
  \begin{equation}
    \label{eq:GNS1}
    - 2 \lambda \| \nabla h \|_2^2
    \\
    \leq
    - \lambda \| \nabla h \|_2^2
    \leq
    -\frac{\lambda}{C_2} \| h \|_3^{\frac{2}{\theta}} \|  h \|_2^{-\frac{2(1-\theta)}{\theta}}
    + \lambda \|h\|_2^2.
  \end{equation}
  In order to show that the term $\ap{h, N(h)}$ is negligible close to
  equilibrium we need to estimate quantities of the type
  $\int_\Omega h_i h_j h_k \d x$, with $i,j,k \in \{1,2,3,4\}$. We do
  that by Hölder's inequality:
  \begin{equation}
    \label{eq:hi-hj-bound-2}
    \left| \int_\Omega h_i h_j h_k \dx \right|
    \leq
    C_3
    \| h \|_3^3,
  \end{equation}
  for a certain constant $C_3$ depending only on the equilibrium
  values $(a_{i,\infty})_{i=1,\dots,4}$. Consequently,
  \begin{equation}
    \label{eq:nonlinear-bound-2}
    \ap{h, N(h)} \leq C_4 \|h\|_3^3.
  \end{equation}
  Using \eqref{eq:GNS1} and \eqref{eq:nonlinear-bound-2} in
  \eqref{eq:1},
  \begin{align*}
    \ddt \| h \|_2^2
    &\leq
    - \lambda \|h\|_2^2
    - \frac{\lambda}{C_2} \| h \|_3^{\frac{12}{d}} \| h \|_2^{\frac{2d-12}{d}}
      + 2 C_4 \|h\|_3^3
    \\
    & =
      - \lambda \|h\|_2^2
      -  \| h \|_3^{\frac{12}{d}}
      \left(
      \frac{\lambda}{C_2} \| h \|_2^{\frac{2d-12}{d}}
      -2 C_4 \| h \|_3^{\frac{3d-12}{d}}
      \right)
    \\
    &\leq
      - \lambda \|h\|_2^2
      -  \| h \|_3^{\frac{12}{d}}
      \left( \frac{\lambda}{C_2}  \| h \|_2^{\frac{2d-12}{d}}
      - C_5 \| h \|_2^{\frac{3d-12}{d}}\right)
    \\
    &=
      - \lambda \|h\|_2^2
      -  \| h \|_3^{\frac{12}{d}}\| h \|_2^{\frac{2d-12}{d}}
      \left(\frac{\lambda}{C_2}  - C_5 \| h \|_2\right)
      ,
  \end{align*}
  where in the second inequality we used that $3d-12\le 0$ and $\|h\|_2
  \leq C \|h\|_3$ for some constant $C$. If it is initially true that
  $$\| h_0 \|_2 \le \frac{\lambda}{C_2 C_5}$$
  then $\|h\|_2$ is decreasing in time and we obtain the result.
\end{proof}

\begin{rem}
  One may wonder if choosing a norm $\|h\|_p$ different from $\|h\|_3$
  may give an improvement in the above argument, perhaps making it
  work in dimensions $d \geq 4$. The authors have tried this with
  other choices of $p$, and in particular for $p = 2_* = (2d)/(d-2)$
  the Poincaré dual of $2$. The dimension $d=4$ seems to be a
  limitation of the argument and not of the specific exponents used in
  the proof.
\end{rem}

\section{Regularity of close-to-equilibrium solutions}
\label{sec:regularity}

We first extend the $L^2$ estimates of the previous section to $L^p$
estimates for a certain range of $p > 2$.

\begin{prp}
  \label{prp:Lp-bounds}
  Assume Hypotheses \ref{hyp:Omega}, \ref{hyp:diffusion>0} and
  \ref{hyp:data}, and let $d \leq 4$. Let $(a_i)_{i=1,\dots,4}$ be a
  solution to the system \eqref{eq:4spec1-2}--\eqref{eq:4spec2-2}. Let
  $a_{i,\infty}$, $i = 1, \dots, 4$, denote the only positive
  equilibrium of \eqref{eq:4spec1} with the same invariants as
  $(a_{i,0})_{i = 1, \dots, 4}$ (that is, satisfying
  \eqref{eq:equilibrium}).

  Assume also that $\|a_{i,0}\|_p < \infty$ for some
  $2 \leq p < \infty$. Then there exist positive constants
  $\lambda, K, \epsilon > 0$ depending on $p$, $d$, $(d_i)_{i=1,\dots,4}$,
  $(a_{i,\infty})_{i = 1, \dots, 4}$, $\|a_0\|_p$ and $\Omega$, such
  that
  \begin{equation*}
    \sum_{i=1}^4 \int_\Omega \frac{|a_i(t,x) -
      a_{i,\infty}|^p}{a_{i,\infty}} \dx
    \leq
    K (1+t)
  \end{equation*}
  whenever the $a_{i,0}$ satisfy
  \begin{equation*}
    \sum_{i=1}^4 \int_\Omega \frac{|a_{i,0}(x) -
      a_{i,\infty}|^2}{a_{i,\infty}} \dx
    \leq
    \epsilon.
  \end{equation*}
\end{prp}

\begin{proof}
  For $d=1$ the result is given directly by Lemma
  \ref{lem:heat-regularisation} (the time dependence can be checked to
  be less than $1 + t$). For $d \geq 2$ and any $p > 1$ we compute
  \begin{align}
    \label{eq:p-norm}
    \ddt \| h \|_p^p
    &= \ddt \sum_{i=1}^4 \int_\Omega \frac{\vert h_i\vert^p}{a_{i,\infty}} \d t
      \nonumber
    \\
    &=
      - \frac{p-1}{4p} \sum_{i=1}^4d_i 
      \int_\Omega \frac{\vert\nabla h_i^{p/2}\vert^2}{a_{i,\infty}} \d x
      \nonumber
    \\
    &+ p \sum_{i=1}^4\frac{1}{a_{i,\infty}}
      \int_\Omega \sign(h_i) \vert h_i \vert^{p-1}\left( L_i(h)+N_i(h)\right) \d x.
  \end{align}
  Using that for any $i,j=1,\dots,4$
  $$
  \int_\Omega |h_i|^{p-1}|h_j| \ dx
  \le C_1 \| h \|_p^p,
  $$
  the linear term corresponding to $L$ is controlled by the $L^p$ norm:
  \begin{equation}
    \label{eq:p-normL}
    p \sum_{i=1}^4\frac{1}{a_{i,\infty}}
    \int_\Omega \sign(h_i) \vert h_i \vert^{p-1} L_i(h) \ dx 
    \le
    C_L \| h \|_p^p. 
  \end{equation}
  The nonlinear term can be bounded by an application of Hölder's
  inquality similar to the one in the previous section, as
  $$
  \int_\Omega |h_i||h_j||h_k|^{p-1} \ dx \le C_2 \| h \|_{p+1}^{p+1},
  $$
  for any $i,j,k \in \{1,2,3,4\}$, obtaining
  \begin{equation}
    \label{eq:p-normN}
    p \sum_{i=1}^4\frac{1}{a_{i,\infty}}
    \int_\Omega \sign(h_i) \vert h_i \vert^{p-1} N_i(h) \ dx 
    \le
    C_N \| h \|_{p+1}^{p+1}.
  \end{equation}
  Using \eqref{eq:p-normL} and
  \eqref{eq:p-normN} in \eqref{eq:p-norm} we have
  \begin{equation}
    \label{eq:ddt-Lp-2}
    \ddt \| h \|_p^p
    \leq
    - \frac{p-1}{4p} \sum_{i=1}^4d_i 
    \int_\Omega \frac{\vert\nabla h_i^{p/2}\vert^2}{a_{i,\infty}} \d x
    + C_L \|h\|_p^p + C_N \|h\|_{p+1}^{p+1}.
  \end{equation}
  
  \medskip\noindent
  \textbf{Proof for $d = 3,4$.}
  In dimensions $d > 2$ the diffusion term can be bounded as follows
  using the Gagliardo-Nirenberg inequality \eqref{eq:GNS}:
  \begin{equation}
    \label{eq:p-normDiff}
    - \frac{p-1}{4p} \sum_{i=1}^4d_i 
    \int_\Omega \frac{\vert\nabla h_i^{p/2}\vert^2}{a_{i,\infty}} \d x
    \le
    -C_D \| h \|_{q}^p + C_D \|h\|_p^p,
  \end{equation}
  where $q:=\frac{dp}{d-2}$ (notice that the minimum of the diffusion
  coefficients $d_i$ appears implicitly in the constant $C_D$). Using
  \eqref{eq:p-normDiff} in \eqref{eq:ddt-Lp-2} we have
  \begin{equation}
    \label{eq:ddt-Lp-3}
    \ddt \| h \|_p^p
    \leq
    - C_D \|h\|_q^p 
    + C_3 \|h\|_p^p + C_N \|h\|_{p+1}^{p+1}.
  \end{equation}
  We now use the interpolation
  $$
  \| h \|_p\le \| h \|_1^{1-\mu} \| h \|_{q}^\mu,
  $$
  with
  \begin{equation*}
    \mu = \frac{q(1-p)}{p(1-q)} = \frac{d(1-p)}{d(1-p)-2},
  \end{equation*}
  together with the fact that the
  $L^1$ norm of $h$ is uniformly bounded in time (since the total mass
  is conserved), to obtain
  $$
  \| h \|_p \le C_2 \| h \|_{q}^\mu,
  $$
  This gives
  \begin{equation}
    \label{eq:p-normDiff-2}
    -C_D \| h \|_{q}^p
    \le -C_4 \| h \|_{p}^{p/\mu}.
  \end{equation}
  To control the nonlinear part we may use interpolation, noticing
  that $2\le p+1\le q$ (which holds since $\frac{d-2}{2} \le p$)
  \begin{equation}
    \label{eq:nlin-bound}
    \| h \|_{p+1}^{p+1}\le \| h \|_2^{\theta(p+1)} 
    \| h \|_{q}^{(p+1)(1-\theta)},
  \end{equation}
  where $\theta=\frac{2(q-p-1)}{(q-2)(p+1)}$. In dimension $d \leq 4$
  the exponent of $\|h\|_q$ is smaller than $p$, since in this case
  $2p-q\le 0$ and thus
  $$
  (p+1)(1-\theta)-p=1-\theta(p+1)=1-\frac{2(q-p-1)}{q-2}
  =\frac{2p-q}{q-2}\le 0.
  $$
  Since $p \leq q$ and the domain $\Omega$ is bounded, we have
  $$\|h\|_q^{(p+1)(1-\theta)} \leq C_5 \|h\|_q^p \|h\|_2^{(p+1)(1-\theta) -
    p}$$ and hence from \eqref{eq:nlin-bound} one obtains
  \begin{equation}
    \label{eq:nlin-bound-2}
    \| h \|_{p+1}^{p+1}
    \le
    C_5 \| h \|_2 \| h \|_{q}^{p}.
  \end{equation}
  Finally, using \eqref{eq:p-normDiff-2} and \eqref{eq:nlin-bound-2}
  in \eqref{eq:ddt-Lp-3}, breaking the term $-C_D \|h\|_q^p$ into two
  halves, we have
  \begin{multline}
    \label{eq:ddt-Lp-4}
    \ddt \| h \|_p^p
    \leq
    - \frac12 C_D \|h\|_q^p 
    - \frac12 C_D \|h\|_q^p 
    + C_3 \|h\|_p^p + C_N \|h\|_{p+1}^{p+1}.
    \\
    =
    - \frac12 C_D \|h\|_q^p
    - \frac{1}{2} C_4 \|h\|_p^{p/\mu}
    + C_3 \|h\|_p^p + C_N C_5 \|h\|_2 \|h\|_q^p
    \\
    =
    - \frac{1}{2} C_4 \|h\|_p^{p/\mu}
    + C_3 \|h\|_p^p
    - \|h\|_q^p \left(
      \frac12 C_D - C_N C_5 \|h\|_2
    \right).
  \end{multline}
  If the initial condition satisfies
  \begin{equation*}
    \|h_0\|_2 \leq \frac{C_D}{2 C_N C_5}
  \end{equation*}
  then this is also the case for any $t \geq 0$, and the last term in
  \eqref{eq:ddt-Lp-4} is nonpositive. This shows that
  \begin{equation*}
    \ddt \| h \|_p^p
    \leq
    - \frac{1}{2} C_4 \|h\|_p^{p/\mu}
    + C_3 \|h\|_p^p.
  \end{equation*}
  Since $\mu < 1$, this differential inequality for $\|h\|_p^p$
  shows that $\|h\|_p$ is bounded uniformly in time by a certain
  constant $K$ which may depend also on $\|h_0\|_p$. (Notice that the
  $(1+t)$ factor is not needed in this case).

  \medskip \noindent \textbf{Proof for $d=2$.} Take any
  $2 < p < 4$, and choose $s, q$ with $1 < s < \frac{2q}{p}$. Using
  the Gagliardo-Nirenberg-Sobolev inequality we have
  \begin{equation*}
    \| h \|_{q}^\frac{p}{2}= 
    \| h^\frac{p}{2} \|_{\frac{2q}{p}} 
    \leq C \| \nabla h^\frac{p}{2}\|_2^\theta 
    \| h^\frac{p}{2}\|_s^{1-\theta}
    + C \|h^{\frac{p}{2}}\|_1
    \leq
    C \| \nabla h^\frac{p}{2}\|_2^\theta 
    \ \| h \|_{\frac{p s}{2}}^\frac{p(1-\theta)}{2}
    +
    C \|h\|_p^{\frac{p}{2}}
  \end{equation*}
  with 
  $$
  \frac{p}{2q}=\theta \frac{d-2}{2d}+\frac{1-\theta}{s}.
  $$
  Rearranging the terms and choosing $s = 4/p > 1$ (which implies
  $q > 2$) we have
  \begin{equation}
    \label{eq:diff-bound-d<3}
    - \| \nabla h^\frac{p}{2}\|_2^2
    \leq
    - C_1 
    \| h \|_{2}^\frac{-p(1-\theta)}{\theta}
    \| h \|_q^\frac{p}{\theta}
    + \|h\|_p^p
  \end{equation}
  with
  $$
  \theta=\frac{dp(q-2)}{q\left(2(2-d)+dp\right)}
  .
  $$
  We will use this bound on the diffusive term in
  \eqref{eq:ddt-Lp-2}. In order to control the $\|h\|_p^p$ term in
  \eqref{eq:ddt-Lp-2} we further bound the diffusive term as follows:
  by interpolation,
  \begin{equation*}
    \| h \|_p
    \le
    \| h \|_1^{\frac{q-p}{p(q-1)}}
    \| h \|_q^{\frac{q(p-1)}{p(q-1)}}
    \leq
    C_2
    \| h \|_q^{\frac{q(p-1)}{p(q-1)}},
  \end{equation*}
  where the last inequality is due to the boundedness of
  $\|h\|_1$. Using this in \eqref{eq:diff-bound-d<3} we have
  \begin{equation}
    \label{eq:lin-bound-d<3}
    \| \nabla h^\frac{p}{2}\|_2^2
    \geq
    C_3 
    \| h \|_{2}^\frac{-p(1-\theta)}{\theta}
    \| h \|_p^{\frac{p}{\theta}\frac{p(q-1)}{q(p-1)}}
    \geq
    C_4 
    \| h \|_p^{\frac{p}{\theta}\frac{p(q-1)}{q(p-1)}}
    - \|h\|_p^p,
  \end{equation}
  now due to the boundedness of $\|h\|_2$.

  In order to control the nonlinear term we interpolate in a similar
  way as before, with any $q>p+1$:
  \begin{equation*}
    \| h \|_{p+1}^{p+1}
    \le
    \| h \|_2^{\frac{2(q-(p+1))}{q-2}}
    \| h \|_q^{\frac{p}{\theta}}
    \| h \|_q^{\frac{q(p-1)}{q-2} - \frac{p}{\theta}}.
  \end{equation*}
  Noticing that the exponent of the last term is nonpositive for
  $d \leq 4$, and that $\|h\|_2$ is bounded by a constant times
  $\|h\|_q$, we have
  \begin{equation}
    \label{eq:nlin-bound-d<3-2}
    \| h \|_{p+1}^{p+1}
    \le
    C_5 \| h \|_2^{p+1-\frac{p}{\theta}}
    \| h \|_q^{\frac{p}{\theta}}.
  \end{equation}
  Now we use \eqref{eq:diff-bound-d<3}, \eqref{eq:lin-bound-d<3} and
  \eqref{eq:nlin-bound-d<3-2} in \eqref{eq:ddt-Lp-2}, breaking again
  the diffusive term into two halves:
  \begin{multline*}
    \ddt \| h \|_p^p
    \leq
    - C_6 \| \nabla h^{\frac{p}{2}} \|_{2}^2
    + C_{10} \|h\|_p^p
    + C_N \| h \|_{p+1}^{p+1}
    \\
    \leq
    - C_7 \| h \|_{2}^\frac{-p(1-\theta)}{\theta}
    \| h \|_q^\frac{p}{\theta}
    - C_8 \| h \|_p^{\frac{p}{\theta}\frac{p(q-1)}{q(p-1)}}
    + C_{10} \|h\|_p^p
    + C_9 \| h \|_2^{p+1-\frac{p}{\theta}}
    \| h \|_q^{\frac{p}{\theta}}.
    \\
    =
    - C_8 \| h \|_p^{\frac{p}{\theta}\frac{p(q-1)}{q(p-1)}}
    + C_{10} \|h\|_p^p
    - \| h \|_q^\frac{p}{\theta}
    \left(
      C_7 \| h \|_{2}^\frac{-p(1-\theta)}{\theta}
      -
      C_9 \| h \|_2^{p+1-\frac{p}{\theta}}
    \right)
  \end{multline*}
  The term in parentheses is nonnegative when
  $$
  \| h \|_2^{\frac{p(1-\theta)}{\theta}+p+1-\frac{p}{\theta}} 
  = \| h \|_2\le \frac{C_7}{C_9},
  $$
  which holds for all $t \geq 0$ as long as it holds for the initial
  condition, since $\|h\|_2$ is nonincreasing in time. Assuming this
  we finally have
  \begin{equation*}
    \ddt \| h \|_p^p
    \leq
    - C_8 \| h \|_p^{\frac{p}{\theta}\frac{p(q-1)}{q(p-1)}}
    + C_{10} \|h\|_p^p.
  \end{equation*}
  Using that $\theta \le \frac{p(q-1)}{q(p-1)}$, we see that this
  differential inequality shows that $\|h\|_p$ is uniformly bounded in
  time. Since this is valid for $2 \leq p < 4$, this shows that all
  the nonlinear terms in \eqref{eq:around-eq} are uniformly bounded in
  $L^s(\Omega)$ for $1 \leq s < 2$ (and of course the linear ones are
  as well). Lemma \ref{lem:heat-regularisation} gives the result for
  all $p$.
\end{proof}

\medskip We are finally in position to complete the proof of Theorem
\ref{thm:exponential-Lp}:

\begin{proof}[Proof of Theorem \ref{thm:exponential-Lp}]
  Theorem \ref{thm:L2-convergence} and Proposition \ref{prp:Lp-bounds}
  give the $L^p$ decay since
  $$\| h \|_p \le \| h \|_2^\theta \| h \|_{p+1}^{1-\theta} \le C \| h
  \|_2^\theta\le \tilde{C} e^{-\lambda \theta t}.$$
  If we assume that $a_{i,0} \in L^\infty(\Omega)$ for $i = 1,\dots,4$
  then this implies that all the nonlinear terms of equation
  \eqref{eq:4spec1-2} are in $L^p(\Omega)$, for all
  $1 \leq p < \infty$. Lemma \ref{lem:heat-regularisation} then shows
  that the $a_i$ are bounded in $L^\infty(\Omega)$ (using the decay of
  the $L^p$ norms, one may check the proof to see that the $L^\infty$
  bound does not depend on time). Classical bootstrap arguments then
  show that the solution $(a_i)_{i=1,\dots,4}$ is in fact infinitely
  differentiable for $t > 0$.
\end{proof}

\section{General quadratic systems with detailed balance}
\label{sec:general}

The arguments in the previous sections can be adapted to any
reaction-diffusion system with at most quadratic nonlinearities, with
the caveat that the constant $\lambda$ giving the speed of convergence
in Theorem \ref{thm:exponential-Lp} cannot be estimated in a
constructive way. One can however show that the linearised system
still has a positive spectral gap (with no estimate on its size) and
relate the speed $\lambda$ to the size of this spectral gap.

\subsection{General setting and main result}
\label{sec:general-setting}

Let us first describe the setting in which our result holds. We
consider a system of $I \geq 2$ species, denoted $A_1, \dots, A_I$,
undergoing a number $R \geq 1$ of different reactions:
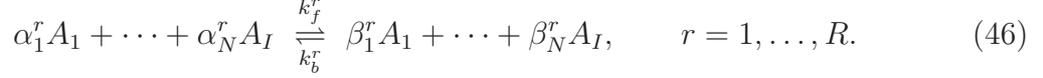
\begin{equation}
  \label{eq:reactions-general}
  \alpha_1^r A_1 + \dots + \alpha_N^r A_I
  \
  \underset{k_b^r}{\overset{k_f^r}{\rightleftharpoons}}
  \
  \beta_1^r A_1 + \dots + \beta_N^r A_I,
  \qquad
  r = 1, \dots, R.
\end{equation}
The description of these processes requires somewhat heavy notation,
which makes clear the reason why we have chosen to present our
results first for the four-species system. We follow the clear
presentation given in \citet{FellnerTang}. The positive numbers
$k_f^r$ and $k_b^r$, for $r=1 \dots, R$, denote the forward and
backward reaction rates, respectively, for each of the $R$
reactions. The vectors $\alpha^r = (\alpha_1^r, \dots, \alpha_I^r)$
and $\beta^r = (\beta_1^r, \dots, \beta_I^r)$ are the
\emph{stoichiometric coefficients} which specify the number of
particles of each species that take part in each reaction. Let us
denote by $a_i = a_i(t,x)$ the concentration of the species $i$ at time
$t$ and position $x \in \Omega$ (for $i = 1, \dots, I$). If the
diffusion coefficient of the species $A_i$ is $d_i > 0$ then the
evolution equation that describes the concentrations
$a = (a_i)_{i = 1,\dots, I}$ is
\begin{equation}
  \label{eq:rd-general}
  \partial_t a_i = d_i \Delta a_i - R_i(a),
  \qquad
  i = 1, \dots, I
\end{equation}
with no-flux boundary conditions as before,
\begin{equation}
  \label{eq:rd-general-boundary}
 \nabla_x a_i (t,x) \cdot \nu (x) = 0,
  \qquad t > 0, \ x \in \partial \Omega, \ i = 1, \dots, I,
\end{equation}
where $\mathbf{R}(a) = (R_1(a), \dots, R_I(a))$ is given by
\begin{equation*}
  \mathbf{R}(a) = \sum_{r=1}^R ( k_f^r a^{\alpha^r} - k_b^r a^{\beta^r} )
  (\alpha^r - \beta^r).
\end{equation*}
Here we use the multiindex notation to write
\begin{equation*}
  a^{\alpha} = \prod_{i=1}^I a_i^{\alpha_i},
  \qquad r = 1,\dots, R,
\end{equation*}
for any vector $\alpha$ with $I$ components. We also set an initial
condition for the $a_i$:
\begin{equation}
  \label{eq:initial-condition-general}
  a_i(0,x) = a_i^0(x), \qquad
  x \in \Omega,\quad
  i = 1, \dots, I.
\end{equation}
One can find the conserved quantities of this system as follows:
define the $R \times I$ matrix $W$ by writing the vectors
$\beta^r - \alpha^r$ as rows:
\begin{equation*}
  W :=
  \left(
    \begin{matrix}
      \beta^1 - \alpha^1
      \\
      \vdots
      \\
      \beta^R - \alpha^R
    \end{matrix}
  \right),
\end{equation*}
and take a matrix $\mathbb Q$ whose rows are a basis of
$(\operatorname{Im} W^\top)^{\top} = \operatorname{Ker} W$ (the
orthogonal complement of the image of $W^\top$ in the usual scalar
product or, equivalently, the kernel of $W$). This ensures that
$\mathbb{Q} W^\top = 0$, and since we can write
\begin{equation*}
  \mathbf{R}(a) = W^\top \mathbf{K}(a),
\end{equation*}
where
$\mathbf{K}(a) := (k_f^r a^{\alpha^r} - k_b^r
a^{\beta^r})_{r=1,\dots,R}$, , we see that the vector
\begin{equation*}
  \int_\Omega \mathbb{Q} a(t,x) \dx
\end{equation*}
is conserved along the evolution \eqref{eq:rd-general} (see
\citet{FellnerTang} for details). An \emph{equilibrium} of the system
\eqref{eq:rd-general} is a vector
$a_\infty = (a_{1,\infty}, \dots, a_{I,\infty})$ with nonnegative
entries such that $\mathbf{R}(a_\infty) = 0$. A \emph{detailed balance
  equilibrium} of \eqref{eq:rd-general} is a vector
$a_\infty = (a_{1,\infty}, \dots, a_{I,\infty})$ with strictly
positive entries such that
\begin{equation}
  \label{eq:db-condition}
  k_f^r a_\infty^{\alpha^r} - k_b^r a_\infty^{\beta^r} = 0
  \qquad \text{for all $r=1,\dots,R$.}
\end{equation}
For illustration, the case of the four-species system is included in
this setting: it corresponds to $R = 1$, $I=4$,
$\alpha^1 = (1,0,1,0)$, $\beta^1 = (0,1,0,1)$,
\begin{equation*}
  W = (-1, 1, -1, 1)
\end{equation*}
and we may choose
\begin{equation*}
  \mathbb{Q} =
  \left(
    \begin{matrix}
      1 & 1 & 0 & 0
      \\
      1 & 0 & 0 & 1
      \\
      0 & 1 & 1 & 0
    \end{matrix}
  \right),
\end{equation*}
which corresponds to the conserved quantities $M_{12}$, $M_{14}$ and
$M_{23}$ as described in Section \ref{sec:assumptions}. Since in this
case there is only one pair of reactions taking place, every
equilibrium with positive entries is a detailed balance equilibrium.

\medskip

Let us describe our assumptions. Our first set mimics the conditions
we used for the four-species system: we always assume Hypothesis
\ref{hyp:Omega} on the domain $\Omega$ and the following positivity
conditions on the parameters:

\begin{hyp}[Positivity of the constants]
  \label{hyp:positivity}
  $I \geq 2$ and $R \geq 1$ are integers, and the diffusion rates
  $d_1, \dots, d_I$ are strictly positive numbers. The reaction
  constants $k^r_f$, $k^r_b$ are strictly positive numbers for
  $r = 1, \dots, R$. For $r=1, \dots, R$ the vectors
  $\alpha^r, \beta^r \in \R^I$ have nonnegative integer coordinates.
\end{hyp}

We always need to assume that the system has at most quadratic
nonlinearities; that is, that each of the reactions happens with at
most two reactants:

\begin{hyp}[The system is at most quadratic]
  \label{hyp:quadratic}
  The sum of the coordinates of each of the vectors
  $\alpha^r, \beta^r \in \R^I$, $r = 1, \dots, R$, is less or equal
  than $2$. That is,
  \begin{equation*}
    \sum_{i=1}^I \alpha_i^r \leq 2,
    \quad
    \sum_{i=1}^I \beta_i^r \leq 2,
    \qquad \text{for $r = 1, \dots, R$.}
  \end{equation*}
\end{hyp}

The previous hypothesis is just a way of saying that each reaction
must have either (a) one reactant, so the resulting term in
$\mathbf{R}(a)$ is linear, or (b) two reactants, in which the
resulting term is quadratic. These two reactants could be the same
one, so one of the components of $\alpha^r$ or $\beta^r$ can be $2$
(and in that case the rest of them must be $0$).

Given the vectors $\alpha^r$, $\beta^r$ for $r=1,\dots,R$ we can
define a matrix $\mathbb{Q}$ of ``conserved quantities'' as described
at the beginning of this section. We fix such a matrix $\mathbb{Q}$
for the rest of this paper. The following assumption corresponds to
Hypothesis \ref{hyp:data} in this general setting: it says that all
conserved quantities corresponding to the initial data are strictly
positive:

\begin{hyp}[Positivity of the initial condition]
  \label{hyp:general-a0}
  The functions $a_{1,0}, \dots, a_{I,0} \: \Omega \to \R$ are in
  $L^2(\Omega)$, are nonnegative, and satisfy that
  \begin{equation*}
    \int_\Omega \mathbb{Q} a_0(x) \dx
    \quad \text{has strictly positive components},
  \end{equation*}
  where $a_0 := (a_{1,0}, \dots, a_{I,0})$.
\end{hyp}
Finally a crucial assumption is that the system associated to the
reactions \eqref{eq:reactions-general} satisfies the detailed balance
condition:
\begin{hyp}[Detailed balance]
  \label{hyp:db-general}
  The system \eqref{eq:rd-general} has a detailed balance equilibrium
  $a_\infty = (a_{1,\infty}, \dots, a_{I,\infty})$ satisfying the
  condition that
  \begin{equation}
    \label{eq:equilibrium-conservations}
    \mathbb{Q} a_\infty = \int_\Omega \mathbb{Q} a_0(x) \dx.
  \end{equation}
  (That is, such that $\int_\Omega (a_\infty - a_0)$ is in the image
  of $W^\top$).
\end{hyp}
To avoid confusion: we are assuming that there is a detailed balance
equilibrium (one whose components are strictly positive and satisfy
\eqref{eq:db-condition}), but we are not assuming it must be the only
one. In particular, we are not ruling out the existence of ``boundary
equilibria'', that is, equilibria with some components equal to $0$
(satisfying \eqref{eq:equilibrium-conservations}). These boundary
equilibria are an important obstacle in the study of the global
behaviour of reaction equations in general (even without diffusion)
and are closely linked to the \emph{global attractor conjecture}
\citep{Horn1972General}.  Some results on reaction-diffusion systems
with boundary equilibria have been recently investigated in
\citet{DFT2017}. Our statement in this general setting is that
solutions in dimension $d \leq 4$ which are close to a detailed
balance equilibria are regular (classical) and relax exponentially to
equilibrium. In other words, detailed balance equilibria are always
locally asymptotically stable, at least in dimension $4$ and below:

\begin{thm}
  \label{thm:exponential-Lp-general}
  Assume Hypothesis \ref{hyp:Omega} and Hypotheses
  \ref{hyp:positivity}--\ref{hyp:db-general}, and let $d\leq 4$. Let
  $(a_i)_{i=1,\dots,I}$ be a solution to the system
  \eqref{eq:rd-general}--\eqref{eq:initial-condition-general}. Let
  $a_{i,\infty}$, $i = 1, \dots, I$, denote a detailed balance
  equilibrium of \eqref{eq:rd-general} with the same invariants as
  $(a_{i,0})_{i = 1, \dots, I}$ (that is, satisfying
  \eqref{eq:equilibrium-conservations}).

  Then for any $2 \leq p < \infty$ there exist positive constants
  $\lambda, K, \epsilon > 0$ depending on $p$, $d$, $(d_i)_{i=1,\dots,I}$,
  $(a_{i,\infty})_{i = 1, \dots, I}$, $\Omega$ and $\|a_0\|_p$ such
  that
  \begin{equation*}
    \sum_{i=1}^I \int_\Omega \frac{|a_i(t,x) -
      a_{i,\infty}|^p}{a_{i,\infty}} \dx
    \leq
    K e^{-\lambda t}
  \end{equation*}
  whenever the $a_{i,0}$ satisfy
  \begin{equation*}
    \sum_{i=1}^I \int_\Omega \frac{|a_{i,0}(x) -
      a_{i,\infty}|^2}{a_{i,\infty}} \dx
    \leq
    \epsilon.
  \end{equation*}
  As a consequence, if $a_{i,0} \in L^\infty(\Omega)$, then the
  solution to
  \eqref{eq:rd-general}--\eqref{eq:initial-condition-general} is
  uniformly bounded in $L^\infty(\Omega)$, and is a classical solution
  for $t > 0$.
\end{thm}
One important difference with Theorem \ref{thm:exponential-Lp} is that
in this generality the constant $\lambda$ cannot be estimated in any
constructive way. We are able to prove that there exists $\lambda$
which satisfies the theorem, but no explicit way is given to estimate
the value of $\lambda$ based on the parameters of the problem. On the
other hand, one cannot expect an estimate in full generality. It
should be noted that the system \eqref{eq:rd-general} contains as a
particular case the Kolmogorov forward equations for a discrete Markov
process with jump rates $k^r_f$, $k^r_b$. (This corresponds to the
case of homogeneous solutions, with reactions where $\alpha^r$ and
$\beta^r$ have only one nonzero entry, and this entry is equal to $1$,
yielding a linear equation). Estimating the speed of convergence to
equilibrium for detailed balance discrete Markov processes is a broad
problem in its own right and it is out of the scope of this paper.

\medskip The rest of this section contains the proof of Theorem
\ref{thm:exponential-Lp-general}, and we always assume the hypotheses
of the theorem. The first observation towards it is to notice that the
entropy structure is maintained by the detailed balance condition:
using \eqref{eq:db-condition} one sees that a (regular) solution to
\eqref{eq:rd-general}--\eqref{eq:rd-general-boundary} satisfies
\begin{multline}
  \label{eq:entropy-general}
  \ddt \sum_{i=1}^I \int_\Omega
  a_{i,\infty} ( u_i \log u_i - u_i + 1 ) \dx
  \\
  =
  - \sum_{i=1}^I \int_\Omega d_i \frac{|\nabla a_i|^2}{a_i} \dx
  -\sum_{r=1}^R \int_\Omega
  k_f^r a_\infty^{\alpha^r} (u^{\alpha^r} - u^{\beta^r})
  (\log u^{\alpha^r} - \log u^{\beta^r}) \dx,
\end{multline}
where we call $u := (u_1, \dots, u_I)$ with
\begin{equation*}
  u_i \equiv u_i(t,x) := \frac{a_i(t,x)}{a_{i,\infty}}
  \qquad \text{for $i=1,\dots,I$.}
\end{equation*}
We follow the same steps as for the four-species system, indicating
only the differences where the arguments are the same.

\subsection{The linearised system}
\label{sec:linearised-general}

Considering the linearisation of \eqref{eq:rd-general} around the
equilibrium $a_\infty = (a_{i,\infty})_{i=1,\dots,I}$ as in Section
\ref{sec:linearised} at first order we obtain the following evolution
equation for the perturbation $h = (h_i)_{i=1,\dots,I}$ (defined so
that $a_i = a_{i,\infty} + h_i$ for all $i$):
\begin{equation}
  \label{eq:linearised-general}
  \partial_t h_i = d_i \Delta h_i + L_i h,
\end{equation}
still with no-flux boundary conditions
\begin{equation}
  \label{eq:linearised-general-boundary}
 \nabla_x h_i (t,x) \cdot \nu (x) = 0,
  \qquad t > 0, \ x \in \partial \Omega, \ i = 1, \dots, I,
\end{equation}
where the linear operator $L h = (L_i h)_{i=1,\dots,I}$ is given by
\begin{equation}
  \label{eq:L_i-general}
  L_i h = - \sum_{j=1}^I \sum_{r=1}^R k_f^r a_\infty^{\alpha^r}
  (\alpha_j^r - \beta_j^r) (\alpha_i^r - \beta_i^r) \frac{h_j}{a_{j,\infty}},
  \qquad i = 1, \dots, I
\end{equation}
or, using matrix notation,
\begin{equation*}
  L h = - \sum_{r=1}^R k_f^r a_\infty^{\alpha^r}
  (\alpha^r - \beta^r)^{\top} (\alpha^r - \beta^r) v,
\end{equation*}
with
\begin{equation*}
  v \equiv (v_1, \dots, v_i)
  := \left( \frac{h_1}{a_{1,\infty}}, \dots, \frac{h_I}{a_{I,\infty}} \right).
\end{equation*}
We can write \eqref{eq:linearised-general} as
\begin{equation}
  \label{eq:linearised-general-abstract}
  \partial_t h = D h + L h =: T h,
\end{equation}
where $D h = (d_i \Delta h_i)_{i = 1, \dots, I}$. Due to the
expression of $L$ it is clear that $\mathbb{Q}L h = 0$ for all $h$, so
solutions to
\eqref{eq:linearised-general}--\eqref{eq:linearised-general-boundary}
satisfy the same conservation laws as
\eqref{eq:rd-general}--\eqref{eq:rd-general-boundary}:
\begin{equation*}
  \int_\Omega \mathbb{Q} h(t,x) \dx = \int_\Omega \mathbb{Q} h(0,x)
  \dx
  \qquad \text{for all $t \geq 0$.}
\end{equation*}
This linearised system ``inherits'' an entropy structure from
\eqref{eq:entropy-general}. Namely,
\begin{multline}
  \label{eq:entropy-lnearised-general}
  \ddt \sum_{i=1}^I \int_\Omega \frac{1}{a_{i,\infty}} h_i^2
  =
  2 \sum_{i=1}^I \int_\Omega \frac{1}{a_{i,\infty}} h_i T_i h
  \\
  =
  - 2 \sum_{i=1}^I d_i \int_\Omega \frac{|\nabla h_i|^2}{a_{i,\infty}} \dx
  - 2 \sum_{r=1}^R \int_\Omega  k_f^r a_\infty^{\alpha^r}
  \left(
    \sum_{i=1}^I (\alpha_i^r - \beta_i^r) \frac{h_i}{a_{i,\infty}}
  \right)^2 \dx.
\end{multline}
Disregarding for the moment the spatial dependence, we may consider
the operator $L$ defined on the vector space $\R^I$. Since
$\mathbb{Q}Lh = 0$ for all $h$, we may restrict $L$ to the vector
space $X = \operatorname{Im} W^\top$. Our first result is that the
linear operator $L$ has negative spectrum in this space (and as an
immediate consequence, detailed balance equilibria are always
isolated). This is known since at least \citet{Horn1972General}, but
we prove it here for completeness:

\begin{lem}
  Assume the conditions of Theorem
  \ref{thm:exponential-Lp-general}. There exists a constant
  $\lambda > 0$ such that the spectrum of the linear operator
  $L \: \operatorname{Im} W^\top \to \operatorname{Im} W^\top$ defined by
  \eqref{eq:L_i-general} is contained in $(-\infty, -\lambda)$.
\end{lem}

\begin{proof}
  One sees from \eqref{eq:L_i-general} that the matrix of $L$ is
  symmetric in the scalar product associated to the norm
  \begin{equation}
    \label{eq:L2-norm-adapted}
    \|h\|^2 := \sum_{i=1}^I \frac{1}{a_{i,\infty}} h_i^2.
  \end{equation}
  Denote this inner product by $\ap{\cdot, \cdot}$. We emphasise that
  in this proof $\|\cdot\|$ and $\ap{\cdot, \cdot}$ denote a norm and
  inner product in $\R^I$, while in the rest of the section they
  denote a norm in $(L^2(\Omega))^I$ (that is, with an additional
  integral in $\Omega$). Due the symmetry of $L$ in this inner
  product, all eigenvalues of $L$ are real. From the identity
  \begin{multline*}
    \ap{L h,h}
    = \sum_{i=1}^I \frac{1}{a_{i,\infty}} h_i L_i h
    = 
    - \sum_{r=1}^R k_f^r a_\infty^{\alpha^r}
    \left(
      \sum_{i=1}^I (\alpha_i^r - \beta_i^r) \frac{h_i}{a_{i,\infty}}
    \right)^2
    \\
    = 
    - \sum_{r=1}^R k_f^r a_\infty^{\alpha^r}
    \Big(
      \ap { \alpha^r - \beta^r, h}
    \Big)^2
    \leq 0
  \end{multline*}
  we see that all eigenvalues of $L$ are contained in $(-\infty,
  0]$.
  Finally, if there is $h \in \operatorname{Im} W^\top$ such that
  $L h = 0$, the latter identity shows that
  \begin{equation*}
    \ap { \alpha^r - \beta^r, h}
    \quad \text{for all $r=1,\dots,R$}.
  \end{equation*}
  This shows that the vector $h = ({h_i})_{i=1,\dots,I}$ is
  perpendicular to the image of $W^\top$ in
  $\ap{\cdot,\cdot}$. Since $h$ is also in the image of
  $W^\top$, we deduce that $h=0$. Hence $0$ is not an eigenvalue of
  $L$ in $\operatorname{Im} W^\top$, and all eigenvalues of $L$ in
  $\operatorname{Im} W^\top$ must be strictly negative. Since $L$ has
  a finite number of eigenvalues, the statement is proved (with no
  constructive estimate on $\lambda$).
\end{proof}

\begin{rem}
  Hypothesis \ref{hyp:quadratic} is not needed in the above result:
  the linearised operator always has a spectral gap once we avoid all
  conservations of the system. Hypothesis \ref{hyp:quadratic} is
  important later in order to show that the linearisation dominates
  the behaviour of the nonlinear equation.
\end{rem}

Once we have the previous result, all remaining points of the proof of
Theorem \ref{thm:exponential-Lp} for the four-species system are
completely analogous in this general setting. We can write an analogue
of Lemma \ref{lem:gap-strong}, with a completely analogous proof which
we omit:

\begin{lem}
  \label{lem:gap-strong-general}
  Assume the hypotheses of Theorem \ref{thm:exponential-Lp-general}
  and define the linear operator $T$ by \eqref{eq:linearised-general},
  \eqref{eq:L_i-general},
  \eqref{eq:linearised-general-abstract}. Define the norm $\|\cdot\|$
  by \eqref{eq:L2-norm-adapted} and denote by $\ap{\cdot,\cdot}$ the
  associated inner product. There exists $\lambda > 0$ such that
  \begin{equation*}
    \ap{h, Th} \leq -\lambda \|h\|_2^2 -  \lambda \| \nabla h \|_2^2
  \end{equation*}
  for all $h \in (L^2(\Omega))^I$ satisfying
  \begin{equation*}
    \int_\Omega \mathbb{Q} h(x) \dx = 0
  \end{equation*}
  (equivalently, such that
  $\int_\Omega h \in \operatorname{Im} W^\top$). The constant
  $\lambda$ depends only on the dimension $d$, the domain $\Omega$,
  the numbers $a_{i,\infty}$, $i = 1, \dots, I$ and the diffusion
  constants $d_i$, $i = 1, \dots, I$.
\end{lem}

\begin{rem}
  Again, Hypothesis \ref{hyp:quadratic} plays no role in the above
  result.
\end{rem}

\subsection{Local $L^p$ stability}
\label{sec:local-L2-general}

The local $L^2$ stability of the nonlinear system
\eqref{eq:rd-general}--\eqref{eq:rd-general-boundary} can the be
obtained using the same proof as in Section \ref{sec:L2}. If
$a = (a_1, \dots, a_I)$ is a solution to
\eqref{eq:rd-general}--\eqref{eq:rd-general-boundary} under the
assumptions of Theorem \ref{thm:exponential-Lp-general}, calling
$h = (h_1, \dots, h_I)$ with
\begin{equation*}
  h_i := a_i - a_{i,\infty},
  \qquad i = 1, \dots, I,
\end{equation*}
we can always write
\eqref{eq:rd-general}--\eqref{eq:rd-general-boundary} as
\begin{equation*}
  \partial_t h_i = d_i \Delta h_i + L_i h + N_i h,
  \qquad i = 1, \dots, I,
\end{equation*}
where $N_i h$ contains only quadratic monomials in $h_j$,
$j=1, \dots, I$ due to Hypothesis \ref{hyp:quadratic}. All arguments
in Sections \ref{sec:L2} and \ref{sec:regularity} can be reproduced in
this case, obtaining Theorem \ref{thm:exponential-Lp-general}.

\section*{Acknowledgements}

The authors are supported by the Spanish \emph{Ministerio de Economía
  y Competitividad} and the European Regional Development Fund
(ERDF/FEDER), project MTM2014-52056-P. We would like to thank Laurent
Desvillettes, Klemens Fellner and {\fontencoding{T5}Bảo
  Tăng} for very helpful conversations on the paper.


\bibliography{bibliography}

\end{document}